\documentclass{amsart}

\newtheorem{theorem}{Theorem}[section]
\newtheorem{lemma}[theorem]{Lemma}
\newtheorem{corollary}[theorem]{Corollary}
\theoremstyle{definition}
\newtheorem{definition}[theorem]{Definition}
\newtheorem{example}[theorem]{Example}

\theoremstyle{remark}

\numberwithin{equation}{section}

\begin{document}

\title{ Rigidity results on gradient Schouten solitons
}

\author{Romildo Pina}
\curraddr{Instituto de Matem\'atica e Estat\'istica, Universidade Federal de Goi\'as, Goi\^ania, Brasil, 74001-970}
\email{romildo@ufg.br}

\author{Ilton Menezes}
\email{iltomenezesufg@gmail.com}



\begin{abstract}

	In this paper we consider $\rho$-Einstein solitons  of type  $M= \left(B^n, g^{*}\right) \times (F^m,g_F)$, where $\left(B^n,g^{*}\right)$ is conformal to a pseudo-Euclidean space and invariant under the action of the pseudo-orthogonal group, and  $\left(F^m,g_{F}\right)$ is an Einstein manifold.	 We provide	all the solutions for the gradient Schouten soliton case. Moreover, in the Riemannian case, we prove that if 
		 $M= \left(B^n, g^{*}\right) \times (F^m,g_F)$ is a complete gradient Schouten soliton then $\left(B^{n},g^{*}\right)$ is isometric to $\mathbb{S}^{n-1}\times \mathbb{R}$ and  $F^m$ is a compact Einstein  manifold.

\end{abstract}

\maketitle

\section{Introduction}
In differential geometry, the Ricci flow is an intrinsic geometric flow. It is a process that deforms the metric of a Riemannian manifold in a way formally analogous to the diffusion of heat.

In 1982, R. Hamilton introduced a nonlinear evolution equation for Riemannian metrics with the aim of finding canonical metrics on manifolds (see \cite{SB} or \cite{Ha}). This evolution equation is known as the Ricci flow, and it has since been used widely and with great success, most notably in Perelman's solution of the Poincaré conjecture. Furthermore, several convergence theorems have been established.

Bryant \cite{BRYANT} proved that there exists a 
complete, steady, gradient Ricci soliton that is spherically symmetric for any 
$n\geq 3$, which is known as Bryant's soliton. In the bi-dimensional case, this solution was obtained explicitly and it is known as the Hamilton cigar. 

Recently, Cao-Chen \cite{CAOCHEN} showed that any complete, steady, gradient Ricci soliton that is locally conformally flat up to homothety, is either flat or isometric to Bryant's soliton. Complete, conformally flat, shrinking gradient solitons have been characterized as being quotients of $\mathbb{R}^n$, $S^n$ or $R\times S^{n-1}$ (see \cite{FG}). 

Motivated by the notion of Ricci solitons on a semi-Riemannian manifold $(M^{n},g)$, $n\geq3$, it is natural to consider geometric flows of the following type: 

\begin{equation}\label{Bouguignon}
\frac{\partial}{\partial t}g(t)=-2(Ric-\rho Kg)
\end{equation}
for $\rho\in\mathbb{R}$, $\rho\neq0$, as in \cite{CA}. We call these the  Ricci-Bourguignon flows. We notice that  short time existence for the geometric flows described in \eqref{Bouguignon} is provided
in (\cite{CCD}).  Associated to the flows, we have the following notion of gradient $\rho$-Einstein solitons, which generate self-similar solutions:

\begin{definition} Let $\left(M^{n},g\right), n\geq3$, be a Riemannian manifold and let $\rho\in \mathbb{R},\rho\neq0$. We say that $(M^{n},g)$ is a gradient $\rho-$Einstein soliton if there exists a smooth function $h:M\longrightarrow\mathbb{R}$, such that the metric $g$ satisfies the equation 
	\begin{equation}\label{def. 1}
	Ric_{g}+Hess_{g}h=\rho K_{g}g +\lambda g
	\end{equation}
	for some constant $\lambda\in\mathbb{R}$, where $K_{g}$ is the scalar curvature of the metric $g$.
\end{definition}

A $\rho$-Einstein soliton is said to be shrinking, steady, or expanding  if $\lambda>0$, $\lambda=0$, or $\lambda<0$, respectively. Furthermore, a $\rho$-Einstein soliton is said to be a gradient Einstein soliton, gradient traceless Ricci soliton, and gradient Schouten soliton if $\rho=\frac{1}{2}$, $\rho=\frac{1}{n}$, and $\rho=\frac{1}{2(n-1)}$, respectively.

In \cite{CA}, the authors studied $\rho-$Einstein solitons and obtained important rigidity results, proving that every compact gradient Einstein, Schouten, or traceless Ricci soliton, is trivial. 

In \cite{IL}, the authors considered $\rho$-Einstein solitons that are conformal to a 
pseudo-Euclidean space and invariant under the action of the pseudo-orthogonal group. We proved all the solutions for the gradient Schouten soliton case. Moreover, we proved that if a gradient Schouten soliton is both complete, conformal to a Euclidean metric, and rotationally symmetric, then it is isometric to $\mathbb{S}^{n-1}\times\mathbb{R}$.

The objective of this paper is to generalize the results obtained in \cite{IL}, considering  a $\rho$-Einstein soliton as being product manifolds of type $M= \left(B^n, g^{*}\right), \times \left(F^m,g_F\right)$, where $\left(B^n,g^{*}\right)$ is conformal to a pseudo-Euclidean space and invariant under the action of the pseudo-orthogonal group, $\left(F^m,g_{F}\right)$ is Einstein manifold. In addition, we will consider that the potential function is defined only in $B^n$.

More precisely, let $(\mathbb{R}^{n},g)$ be the standard pseudo-Euclidean space with metric $g$ and coordinates $(x_{1},...,x_{n})$, with $g_{ij}=\delta_{ij}\varepsilon_{i}$, $1\leq i,j \leq n$, where $\delta_{ij}$ is the Kronecker delta, and $\varepsilon_{i}=\pm1$, with at least one $\varepsilon_{i}$ equal to one. Let $r=\sum_{i=1}^{n}\varepsilon_{i}x_{i}^{2}$ be a basic invariant for an $(n-1)-$dimensional
pseudo-orthogonal group. Let us found the necessary and sufficient conditions for the existence of differentiable functions $\psi(r)$ and $h(r)$ such that the metric $\tilde{g}=g^{*}+g_{F}$, where $g^{*}=\frac{g}{\psi^{2}}$ satisfies the equation 
\begin{align}\label{0000}
Ric_{\tilde{g}}+Hess_{\tilde{g}}h=\rho K_{\tilde{g}}\tilde{g} +\tilde{\lambda} \tilde{g}
\end{align}
with $\rho=1/2(n-1)$.

In the proof of the results, we will consider product manifolds of type $\left(\mathbb{R}^n,g^{*}\right)\times F^m$ as being a particular case of the warped product $\left(\mathbb{R}^n,g^{*}\right)\times_{f} F^m$ considering $f \equiv 1$. For more details, see \cite{O'neil}.
 
We initially find a system of differential equations, which functions $h$ and $\psi$ must satisfy, so that the metric $\tilde{g}=g^{*}+g_{F}$ satisfies \eqref{0000} (see Theorem \ref{theorem 1}).

Note that if the solutions are invariant under the action of the pseudo-orthogonal group, the system of partial differential equations given in Theorem \ref{theorem 1} can be transformed into a system of ordinary differential equations (see Corollary \ref{Theorem 2}). 

Making $\rho=\frac{1}{2(n-1)}$, in Lemma \ref{lema2} we found conditions necessary on $\psi(r)$ 
so that the problem allows for a solution. Finally, we obtain the necessary and sufficient conditions on $h(r)$ and $\psi(r)$ for the existence of one gradient Schouten soliton. In this case, all solutions are given explicitly. Moreover, in the Riemannian case, we prove that if 	$\left(\left(B^{n}, g^{*}\right)\times F^m\right)$, $n\geq3$ is a complete gradient Schouten soliton, where $\left(B^n,g^{*}\right)$ is conformal to a Euclidean metric and rotationally symmetric where $g^{*}_{ij}=\delta_{ij}/\psi^{2}$, and $F^m$ is an Einstein Riemannian manifold. Then $\left(B^{n},g^{*}\right)$ is isometric to $\mathbb{S}^{n-1}\times \mathbb{R}$, $F^m$ is a compact manifold.

In what follows, we state our main results. We denote the second order derivative of $\psi$ and $h$ by $\psi_{,x_{i}x_{j}}$ and $h_{,x_{i}x_{j}}$, respectively, with respect to $x_{i}x_{j}$.
\section{Main results}
\begin{theorem} \label{theorem 1}
	Let $(\mathbb{R}^n,g)$ be a pseudo-Euclidean space, $n\geq 3$ with coordinates $x=(x_1,\cdots, x_n)$ and $g_{ij}=\delta_{ij}\varepsilon_i$. Consider the product manifold $M = (\mathbb{R}^{n}, g^{*})\times F^{m}$ with metric tensor  $\widetilde{g} = g^{*} + g_{F}$, where $\displaystyle g^{*} =
	\dfrac{1}{\psi^{2}}g$,  $F$ is an Einstein semi-Riemannian manifold with constant Ricci curvature $\lambda_{F}$, $m\geq 1$, $\psi, h:\mathbb{R}^{n}\rightarrow \mathbb{R}$, are smooth functions. Then $M$ is a gradient $\rho$--Einstein soliton with 
	\begin{equation}\label{0}
	 Ric_{\widetilde{g}} +Hess_{\widetilde{g}}h = \left(\rho K_{\tilde{g}}+\tilde{\lambda}\right) \widetilde{g}, \ \ \ \ \ \rho\in\mathbb{R},
	 \end{equation}
	if, and only if, the functions $\psi$ and $h$ satisty
	
	\begin{equation} \label{dpe1}
	\begin{array}{l}
	(n-2) \psi_{x_ix_j}+\psi h_{,x_ix_j}+\psi_{,x_i}h_{,x_j}+\psi_{,x_j}h_{,x_i}=0,\hspace{0.5cm} \forall  i\neq j,
	\end{array}
	\end{equation}
	\begin{equation} \label{dpe2}
	\begin{array}{l}
	 \varepsilon_i\sum\limits_{k=1}^{n}\varepsilon_k\left[\psi\psi_{,x_kx_k}-(n-1)\left(\psi_{,x_k}\right)^2-\psi\psi_{,x_k}h_{,x_k}-2(n-1)\rho\psi\psi_{,x_kx_k}+(n-1)n\rho \left(\psi_{,x_k}\right)^2\right]\\
	\displaystyle
	+\left[(n-2)\psi\psi,_{x_ix_i}+\psi^2 h,_{x_ix_i}+2\psi\psi_{,x_i} h_{,x_i}\right]=\left(\lambda_F m\rho+\tilde{\lambda}\right)\varepsilon_i, \hspace{0.5cm}
	\forall\hspace{0.5cm} i=1,\ldots,n,
	\end{array}
	\end{equation}
	\begin{equation} \label{dpe3}
	\begin{array}{l}
	 \sum\limits_{k=1}^n\varepsilon_k\left((n-1)n\rho\left(\psi_{,x_k}\right)^2-2(n-1)\rho \psi\psi_{,x_kx_k}\right)=\lambda_{F}\left(m\rho-1\right)+\tilde{\lambda}.
	 \end{array}
	\end{equation}
\end{theorem}
\begin{proof} Suppose initially that $M=\left(\mathbb{R}^n,g^{*}\right)\times F^{m}$ is a gradient $\rho$--Einstein soliton with potential function $h$, that is, 
 \begin{equation}\label{14}
 Ric_{\widetilde{g}} +Hess_{\widetilde{g}}h = \left(\rho K_{\tilde{g}}+\tilde{\lambda}\right) \widetilde{g}, \ \ \ \ \ \rho\in\mathbb{R}.
\end{equation}

Let $X_{1}, X_{2},\ldots, X_{n}\in{\mathcal{L}}(\mathbb{R}^{n})$ and $Y_{1}, Y_{2}, \ldots, Y_{m}\in{\mathcal{L}}(F)$, where ${\mathcal{L}}(\mathbb{R}^{n})$ and ${\mathcal{L}}(F)$  are respectively the spaces of lifts of a vector field on $\mathbb{R}^{n}$ and $F$ to ($\mathbb{R}^{n}\times F^m$ ), then
\begin{equation} \label{Ricprodsemi}
\begin{cases}
Ric_{\widetilde{g}}(X_i,X_j) = \dfrac{(n-2)\psi_{,x_ix_j}}{\psi}, \ \forall \hspace{0.2cm}i\neq j,\hspace{0.2cm} i,j=1,\ldots,n\\
Ric_{\widetilde{g}}(X_i,X_i)=\dfrac{(n-2)\psi_{,x_ix_i}+\varepsilon_i\sum\limits_{k=1}^{n}\varepsilon_k\psi_{,x_kx_k}}{\psi}-(n-1)\varepsilon_i\sum\limits_{k=1}^n\left(\dfrac{\psi_{,x_k}}{\psi}\right)^2, \ \forall\hspace{0.2cm}\hspace{0.2cm} i=1,\ldots,n\\
Ric_{\widetilde{g}}(X_i,Y_j)=0, \ \forall i=1,\ldots,n, j=1,\ldots,m\\
Ric_{\widetilde{g}}(Y_i,Y_j)=\lambda_{F}\widetilde{g}(Y_i,Y_j), \ \forall\hspace{0.2cm} i,j=1,\ldots,m.
\end{cases}
\end{equation}

As the potential function $h:\mathbb{R}^n\longrightarrow\mathbb{R}$ is defined only on the basis, we have
\begin{equation*}
Hess_{\tilde{g}}(h)\left(X_i,X_j\right)=Hess_{g^{*}}(h)\left(X_i,X_j\right)\,\hspace{0.2cm} \forall\hspace{0.2cm} i,j=1,\ldots,n.
\end{equation*}

Therefore,
\begin{equation}\label{Hessprodsemi}
\begin{cases}
Hess_{\widetilde{g}}h(X_i,X_j) = h_{,x_{i}x_{j}}+\frac{\psi_{,x_{j}}h_{,x_{i}}}{\psi}+\frac{\psi_{,x_{i}}h_{,x_{j}}}{\psi}, \ \forall\hspace{0.2cm}i\neq j,\hspace{0.2cm} i,j=1,\ldots,n\\
Hess_{\widetilde{g}}(X_i,X_i)=h_{,,x_{i}x_{i}}+\frac{2\psi_{,x_{i}}h_{,x_{i}}}{\psi}-\varepsilon_{i}\sum_{k=1}^{n}\varepsilon_{k}\frac{\psi_{,x_{k}}h_{,x_{k}}}{\psi}, \ \hspace{0.2cm}\forall\hspace{0.2cm} \hspace{0.2cm} i=1,\ldots,n.
\end{cases}
\end{equation}

We split the proof into two cases. In the first case, we will consider the vector fields on the base and in the second one, on the fiber.

If $X_{1}, X_{2},\ldots, X_{n}\in{\mathcal{L}}(\mathbb{R}^{n})$, then to $i\neq j$, from equation \eqref{14} together with equations \eqref{Ricprodsemi} and \eqref{Hessprodsemi}, follows that 
\begin{equation*}
\dfrac{(n-2)\psi_{,x_ix_j}}{\psi}+h_{,x_ix_j}+\dfrac{\psi_{,x_i}h_{,x_j}}{\psi}+\dfrac{\psi_{,x_j}h_{,x_i}}{\psi}=0.
\end{equation*}
Therefore,
\begin{equation*}
\begin{array}{l}
(n-2) \psi_{x_ix_j}+\psi h_{,x_ix_j}+\psi_{,x_i}h_{,x_j}+\psi_{,x_j}h_{,x_i}=0.
\end{array}
\end{equation*}

For $i=j$ we will need the following facts. It is well known that (see, e.g., \cite{Paul}) 
\begin{equation}\label{scalar}
K_{\tilde{g}}=K_{g^{*}}+K_{F}
\end{equation}
where $K_{\tilde{g}}$, $K_{g^{*}}$ and $K_{F}$ represent the scalar curvatures of $M=\left(\mathbb{R}^n,g^{*}\right)\times F^{m}$,  $\left(\mathbb{R}^n,g^{*}\right)$ and  $F^m$, respectively. Note that 
\begin{equation*}
K_{g^{*}}=\sum\limits_{k=1}^n\varepsilon_k\left[2(n-1)\psi\psi_{,x_kx_k}-(n-1)n\left(\psi_{,x_k}\right)^2\right],\hspace{0.3cm} K_{F}=\lambda_{F}m.
\end{equation*}

Substituting this expression into \eqref{scalar}, we get 
\begin{equation*}
K_{\tilde{g}}=\sum\limits_{k=1}^n\varepsilon_k\left[2(n-1)\psi\psi_{,x_kx_k}-(n-1)n\left(\psi_{,x_k}\right)^2\right]+\lambda_{F}m
\end{equation*}

Multiplying by $\rho$ on both sides of the above equality, we have

\begin{equation}\label{8}
\rho \left(K_{g^{*}}+K_{F}\right)=\rho\lambda_{F}m+\rho\sum\limits_{k=1}^n\varepsilon_k\left[2(n-1)\psi\psi_{,x_kx_k}-(n-1)n\left(\psi_{,x_k}\right)^2\right].
\end{equation}

Replacing the expressions found in \eqref{Ricprodsemi}, \eqref{Hessprodsemi} and \eqref{scalar} in \eqref{14}, we have
\begin{equation}\label{9}
\dfrac{(n-2)\psi_{,x_ix_i}+\varepsilon_i\sum\limits_{k=1}^n\varepsilon_k\psi_{,x_kx_k}}{\psi}-(n-1)\varepsilon_i\sum\limits_{k=1}^n\varepsilon_k\left(\dfrac{\psi_{,x_k}}{\psi}\right)^2+h_{,x_ix_i}+2\dfrac{\psi_{,x_i}h_{,x_i}}{\psi}-\varepsilon_i\sum\limits_{k=1}^n\varepsilon_k\dfrac{\psi_{,x_k}h_{,x_k}}{\psi}
\end{equation}
\begin{equation*}
=\left(\rho \left(K_{g^{*}}+K_{F}\right)+\tilde{\lambda}\right)g^{*}\left(X_i,X_i\right).
\end{equation*}

Substituting the equation \eqref{8} into \eqref{9}
 and multipying by $\psi^2$, we yields
\begin{equation*} 
\begin{array}{l}
 \varepsilon_i\sum\limits_{k=1}^{n}\varepsilon_k\left[\psi\psi_{,x_kx_k}-(n-1)\left(\psi_{,x_k}\right)^2-\psi\psi_{,x_k}h_{,x_k}-2(n-1)\rho\psi\psi_{,x_kx_k}+(n-1)n\rho \left(\psi_{,x_k}\right)^2\right]\\
\displaystyle
+\left[(n-2)\psi\psi,_{x_ix_i}+\psi^2 h,_{x_ix_i}+2\psi\psi_{,x_i} h_{,x_i}\right]=\left(\lambda_F m\rho+\tilde{\lambda}\right)\varepsilon_i.
\end{array}
\end{equation*}

\item If  $Y_{1}, Y_{2}, \ldots, Y_{m}\in{\mathcal{L}}(F)$, we have

\begin{equation}\label{10}
Hess_{\tilde{g}}(h)\left(Y_i,Y_j\right)=0
\end{equation}

Substituting the equations given by \eqref{Ricprodsemi}, \eqref{8} and \eqref{10} into \eqref{14}, we obtain

\begin{equation}\label{11}
\left(\rho\lambda_{F}m+\rho\sum\limits_{k=1}^n\varepsilon_k\left(2(n-1)\psi\psi_{,x_kx_k}-(n-1)n\left(\psi_{,x_k}\right)^2\right)+\tilde{\lambda}\right)g_{F}\left(Y_i,Y_j\right)=\lambda_{F}g_{F}\left(Y_i,Y_j\right),
\end{equation}
equivalently, 
\begin{equation*}
\rho\lambda_{F}m+\rho\sum\limits_{k=1}^n\varepsilon_k\left(2(n-1)\psi\psi_{,x_kx_k}-(n-1)n\left(\psi_{,x_k}\right)^2\right)+\tilde{\lambda}=\lambda_{F},
\end{equation*}
which implies
\begin{equation*}
\sum\limits_{k=1}^n\varepsilon_k\left((n-1)n\rho\left(\psi_{,x_k}\right)^2-2(n-1)\rho \psi\psi_{,x_kx_k}\right)=\lambda_{F}\left(m\rho-1\right)+\tilde{\lambda}.
\end{equation*}
\end{proof}

We want to find solutions to the system of equations (\ref{dpe1}),
(\ref{dpe2}) and (\ref{dpe3})
of the form $\psi(r)$ and $h(r)$, where   $r=\sum\limits_{k=1}^n\varepsilon_kx_{k}^2$. The following theorem reduces the system of partial differential equations (\ref{dpe1}),
(\ref{dpe2}) and (\ref{dpe3}) into a system of ordinary differential equations that must be satisfied by such solutions.

\begin{corollary}\label{Theorem 2}
	Let $( \mathbb{R}^n, g)$ be a
	pseudo-Euclidean space, $n\geq 3$, with coordinates
	$x=(x_1,\cdots, x_n)$ and $g_{ij}=\delta_{ij}\varepsilon_i$.
	Consider $M = (\mathbb{R}^{n}, g^{*})\times F^{m}$,
	where $\displaystyle g^{*} = \frac{1}{\psi^{2}}g$, $F^{m}$ an Einstein semi--Riemannian  manifold
	with constant Ricci curvature $\lambda_{F}$ and smooth
	functions $\psi(r)$ and $h(r)$, where
	$r=\sum\limits_{k=1}^n\varepsilon_kx_{k}^2$. Then $M$ is a
	gradient $\rho$--Einstein soliton with $h$ as a potential function if, and only if,
	the functions  $h$ and $\psi$ satisfy:
\begin{equation}\label{15}
(n-2)\psi''+\psi h''+2\psi' h'=0,
\end{equation}

\begin{equation}\label{16}
2\psi \left[2(n-1)(1-n\rho)\psi'+\psi h'\right]+
\end{equation}
\begin{equation*}
4r\left[\left(1-2(n-1)\rho\right)\psi\psi''+(n-1)(n\rho-1)\left(\psi'\right)^2-\psi\psi'h'\right]=\left(\lambda_{F}m\rho+\tilde{\lambda}\right).
\end{equation*}
and
\begin{equation}\label{17}
-4n(n-1)\rho\psi\psi'+4r\left[(n-1)n\rho\left(\psi'\right)^2-2(n-1)\rho\psi \psi''\right]=\lambda_{F}\left(m\rho-1\right)+\tilde{\lambda}.
\end{equation}
\end{corollary}
\begin{proof}Let $g^{*}=\psi^{-2}g$ be a conformal metric of $g$. We are assuming that $\psi(r)$
	and $h(r)$ are functions of $r$, where $r=\sum_{k=1}^{n}\varepsilon_{k}x_{k}^{2}$. Hence, we have
	
	\begin{equation*}
	\psi_{,x_{i}}=2\varepsilon_{i}x_{i}\psi',\hspace{0,5cm} \psi_{,x_{i}x_{i}}=4x_{i}^{2}\psi''+2\varepsilon_{i}\psi', \hspace{0,5cm} \psi_{,x_{i}x_{j}}=4\varepsilon_{i}\varepsilon_{j}x_{i}x_{j}\psi'',
	\end{equation*}
	and
	\begin{equation*}
	h_{,x_{i}}=2\varepsilon_{i}x_{i}h',\hspace{0,5cm} h_{,x_{i}x_{i}}=4x_{i}^{2}h''+2\varepsilon_{i}h', \hspace{0,5cm} h_{,x_{i}x_{j}}=4\varepsilon_{i}\varepsilon_{j}x_{i}x_{j}h''.
	\end{equation*}
Substituting these expressions into the first equation of Theorem \ref{theorem 1}, we get 
\begin{equation*}
(n-2)\left(4\varepsilon_i\varepsilon_j x_ix_j\psi''\right)+\psi\left(4\varepsilon_i\varepsilon_j x_ix_jh''\right)
+\left(2\varepsilon_ix_i\psi'\right)\left(2\varepsilon_jx_jh'\right)+\left(2\varepsilon_jx_j\psi'\right)\left(2\varepsilon_ix_ih'\right)=0,
\end{equation*}
equivalently,
\begin{equation*}
4\varepsilon_i\varepsilon_j\left[(n-2)\psi''+\psi h''+2\psi' h'\right]x_ix_j=0.
\end{equation*}

Since there exists $i\neq j$, such that $x_ix_j\neq 0$, it follows that 
\begin{equation*}
(n-2)\psi''+\psi h''+2\psi' h'=0
\end{equation*}

Similarly, considering the second equation of Theorem \ref{theorem 1}, we obtain
\begin{equation*}
\psi \left[(n-2)\left(4x_i^2\psi''+2\varepsilon_i\psi'\right)+\psi \left(4x_i^2h''+2\varepsilon_ih'\right)+2\left(2\varepsilon_ix_i\psi'\right)\left(2\varepsilon_ix_ih'\right)\right]+
\end{equation*}
\begin{equation*}
\varepsilon_i\sum\limits_{k=1}^n\varepsilon_k\left[\left(1-2(n-1)\rho\right)\psi\left(4x_k^2\psi''+2\varepsilon_k\psi'\right)+(n-1)(n\rho-1)\left(2\varepsilon_kx_k\psi'\right)^2\right]+
\end{equation*}
\begin{equation*}
\varepsilon_i\sum\limits_{k=1}^n\varepsilon_k\left[-\psi\left(2\varepsilon_kx_k\psi'\right)\left(2\varepsilon_kx_kh'\right)\right]
=\left(\lambda_{F}m\rho+\tilde{\lambda}\right)\varepsilon_i,
\end{equation*}
equivalently,
\begin{equation*}
4\varepsilon_i\sum\limits_{k=1}^n\varepsilon_kx_k^2\left[\left(1-2(n-1)\rho\right)\psi\psi''+(n-1)(n\rho-1)\left(\psi'\right)^2-\psi\psi'h'\right]+
\end{equation*}
\begin{equation*}
4\psi \left[(n-2)\psi''+\psi h''+2\psi'h'\right]x_i^2+2\psi \left[(n-2)\psi'+\psi h'\right]\varepsilon_i+
\end{equation*}
\begin{equation*}
2\varepsilon_i\sum\limits_{k=1}^n\varepsilon_k\left[\left(1-2(n-1)\rho\right)\psi\psi'\varepsilon_k\right]=\left(\lambda_{F}m\rho+\tilde{\lambda}\right)\varepsilon_i.
\end{equation*}
From equation \eqref{15}, it followos that 
\begin{equation*}
2\psi \left[(n-2)\psi'+\psi h'+\left(1-2(n-1)\rho\right)n\psi'\right]+
\end{equation*}
\begin{equation*}
+4r\left[\left(1-2(n-1)\rho\right)\psi\psi''+(n-1)(n\rho-1)\left(\psi'\right)^2-\psi\psi'h'\right]=\left(\lambda_{F}m\rho+\tilde{\lambda}\right),
\end{equation*}
which implies,
\begin{equation*}
2\psi \left[2(n-1)(1-n\rho)\psi'+\psi h'\right]+
\end{equation*}
\begin{equation*}
4r\left[\left(1-2(n-1)\rho\right)\psi\psi''+(n-1)(n\rho-1)\left(\psi'\right)^2-\psi\psi'h'\right]=\left(\lambda_{F}m\rho+\tilde{\lambda}\right).
\end{equation*}

Finally, from the last equation of Theorem \ref{theorem 1}, we have 
\begin{equation*}
\sum\limits_{k=1}^n\varepsilon_k\left[-2(n-1)\rho\psi \left(4x_k^2\psi''+2\varepsilon_k\psi'\right)+(n-1)n\rho\left(2\varepsilon_kx_k\psi'\right)^2\right]=\lambda_{F}(m\rho-1)+\tilde{\lambda},
\end{equation*}
equivalently,

\begin{equation*}
-4n(n-1)\rho\psi\psi'+4r\left[(n-1)n\rho\left(\psi'\right)^2-2(n-1)\rho\psi \psi''\right]=\lambda_{F}\left(m\rho-1\right)+\tilde{\lambda},
\end{equation*}
therefore, the proof is done.
\end{proof}

When $\rho=\frac{1}{2(n-1)}$, we obtain the following result:
\begin{corollary}\label{lemma1}
	Let $( \mathbb{R}^n, g)$ be a
pseudo-Euclidean space, $n\geq 3$, with coordinates
$x=(x_1,\cdots, x_n)$ and $g_{ij}=\delta_{ij}\varepsilon_i$.
Consider $M = (\mathbb{R}^{n}, g^{*})\times F^{m}$,
where $\displaystyle g^{*} = \frac{1}{\psi^{2}}g$, $F^{m}$ is an Einstein semi--Riemannian  manifold
with constant Ricci curvature $\lambda_{F}$ and smooth
functions $\psi(r)$ and $h(r)$, where
$r=\sum\limits_{k=1}^n\varepsilon_kx_{k}^2$. Then $M$ is a
	gradient Schouten soliton with $h$ as a potential function if, and only if,
	the functions $h$ and $\psi$ satisfy:
	\vspace{12pt}
	\begin{equation}\label{18}
		(n-2)\psi''+\psi h''+2\psi'h'=0,
	\end{equation}
	and 
	\begin{center}
		\begin{equation}\label{19}
			\psi\left[\left(n-2\right)\psi'+\psi h'\right]
			+r\left[\left(2-n\right)(\psi')^{2}-2\psi \psi'h'\right]=\frac{\lambda_{F}m}{4(n-1)}+\frac{\tilde{\lambda}}{2},
		\end{equation}
	\end{center}
	and 
	\begin{center}
		\begin{equation}\label{20}
		-n\psi\psi'+2r\left[\frac{n}{2}\left(\psi'\right)^2-\psi\psi''\right]=\frac{\lambda_{F}}{4(n-1)}\left(m-2n+2\right)+\frac{\tilde{\lambda}}{2}.
		\end{equation}
	\end{center}
\end{corollary}
\begin{proof} By simply inserting $\rho=\frac{1}{2(n-1)}$ into equations \eqref{19} and \eqref{20}, the result follows.
\end{proof}	

The next lemma provides us with the necessary conditions for the existence of gradient Schouten solitons.

\begin{lemma}\label{lema2}
	Consider smooth functions $\psi(r)$ and $h(r)$ with $r=\sum_{i=1}^{n}\varepsilon_{i}x_{i}^{2}$. If $\psi(r)$ and $h(r)$ satisfy \eqref{18} and \eqref{19}, then $\psi(r)$ satisfies the following ordinary differential equation:  
	\begin{equation}
	r\psi\psi''=r\left(\psi'\right)^{2}-\psi\psi'.
	\end{equation}
\end{lemma}
\begin{proof} Differentiating both sides of equation \eqref{19}, we get
	\begin{align*}
	\psi'\left[(n-2)\psi'+\psi h'\right]+\psi\left[(n-2)\psi''+\psi' h'+\psi h''\right]+\left[(2-n)\left(\psi'\right)^{2}-2\psi\psi'h'\right]+\\+r\left\{2(2-n)\psi'\psi''-2\left[\left(\psi'\right)^{2}h'+\psi\psi''h'+\psi\psi'h''\right]\right\}=0.
	\end{align*}
	
	This is equivalent to
	\begin{equation}\label{21}
	(n-2)\left(\psi'\right)^{2}+\psi\psi'h'+\left(n-2\right)\psi\psi''
	+\psi\psi'h'+\psi^{2}h''+\left(2-n\right)\left(\psi'\right)^{2}-2\psi\psi'h'+\\
	\end{equation}
	\begin{equation*}
	+2\left(2-n\right)r\psi\psi''-2r\left(\psi'\right)^{2}h'-2r\psi\psi''h'-2r\psi\psi'h''=0.
	\end{equation*}
	
	From \eqref{18} we obtain
	\begin{equation}\label{22}
	\psi h''=-\left[(n-2)\psi''+2\psi'h'\right].
	\end{equation}
	
	Substituting \eqref{22} into \eqref{21}, we obtain 
	\begin{align*}
	(n-2)\psi\psi''-\psi\left[(n-2)\psi''+2\psi'h'\right]+2(2-n)r\psi'\psi''-2r\left(\psi'\right)^{2}h'-\\-2r\psi\psi''h'+2r\psi'\left[(n-2)\psi''+2\psi'h'\right]=0,
	\end{align*}
	which is equivalent to
	\begin{align*}
	2\psi\psi'h'-2r\left(\psi'\right)^{2}h'+2r\psi\psi''h'=0.
	\end{align*}
	
	How $h ' \ne 0$, we have that
	\begin{align*}
	r\psi\psi''=r(\psi')^{2}-\psi \psi'.
	\end{align*}
	
	This concludes the proof.
\end{proof}

As a consequence of Lemma \ref{lema2} we are going to prove a classification theorem for gradient schouten sotitons of type $M = (\mathbb{R}^{n}, g^{*})\times F^{m}$, where  $(\mathbb{R}^{n}, g^{*})$ is  conformal to a pseudo-Euclidean space and invariant by the action of the pseudo-orthogonal group and $F^m$ is an Einstein manifold.

\begin{theorem}\label{theorem 4}
	Let $( \mathbb{R}^n, g)$ be a
	pseudo-Euclidean space, $n\geq 3$, with coordinates
	$x=(x_1,\cdots, x_n)$ and $g_{ij}=\delta_{ij}\varepsilon_i$.
	Consider $M = (\mathbb{R}^{n}, g^{*})\times F^{m}$ a semi--Riemannian manifold product,
	where $\displaystyle g^{*} = \frac{1}{\psi^{2}}g$, $F^{m}$ is an  Einstein semi--Riemannian  manifold
	with $m\geq2$, constant Ricci curvature $\lambda_{F}$ and smooth
	functions $\psi(r)$ and $h(r)$, where
	$r=\sum\limits_{k=1}^n\varepsilon_kx_{k}^2$. Then $M$ is a
	gradient Schouten soliton with $h$ as a potential function if, and only if,
	the functions $h$ and $\psi$ satisfy:
	
	\begin{equation}\label{23}
	\left\{\begin{array}{lcl}
	\psi(r)=k_2r,\\
	h(r)=\frac{\lambda_{F}}{2k_2^2}r^{-1}+k_1,\\
	\tilde{\lambda}=-\frac{\lambda_{F}}{2(n-1)}\left(m-2n+2\right),
	\end{array}\right.
	\end{equation}
	and 
	\begin{equation*}
	\left\{\begin{array}{lcl}
	\psi(r)=k_2r^{\frac{1}{2}},\\
	h(r)=\frac{(n-2)}{8}\left(\ln r\right)^{2}+c\ln r + c_1,\\
	\tilde{\lambda}=-\frac{(m-n+1)(n-2)}{2(n-1)}k_2^2,\\
	\lambda_{F}=(n-2)k_2^2,
	\end{array}\right.
	\end{equation*}
where $c$, $c_{1}$ and $k_1\in \mathbb{R}$ and $k_2\in\mathbb{R}_{+}^{*}$.
\end{theorem}
\begin{proof} From Lemma \ref{lema2} we know that $\psi(r)=k_2r^s$ with $ k_2 >0$. Since $\psi(r)$ must satisfy equation \eqref{20}, it follows that 
	\begin{equation*}
	-nk_2r^sk_2sr^{s-1}+nrk_2^2s^2r^{2s-2}-2rk_2r^sk_2s(s-1)r^{s-2}=\frac{\lambda_{F}}{4(n-1)}\left(m-2n+2\right)+\frac{\tilde{\lambda}}{2},
	\end{equation*}
	equivalently,
	\begin{equation*}
	-nk_2^2sr^{2s-1}+nk_2^2s^2r^{2s-1}-2k_2^2s(s-1)r^{2s-1}=\frac{\lambda_{F}}{4(n-1)}\left(m-2n+2\right)+\frac{\tilde{\lambda}}{2},
	\end{equation*}
which implies 
	\begin{equation}\label{27}
	s(s-1)(n-2)k_2^2r^{2s-1}=\frac{\lambda_{F}}{4(n-1)}\left(m-2n+2\right)+\frac{\tilde{\lambda}}{2}.
	\end{equation}

Since the right side must be constant, we have either $s=1$ or $s=\frac{1}{2}$.

If $s=1$, we have $\tilde{\lambda}=-\frac{\lambda_{F}}{2(n-1)}\left(m-2n+2\right)$, by the fact that $\psi(r)$ and $h$ satisfy equation \eqref{19}. Note that
\begin{equation*}
\psi(r)=k_2r,\hspace{0.5cm}\psi'(r)=k_2,\hspace{0.5cm}\psi''(r)=0.
\end{equation*}
Consequently 
\begin{equation*}
	(n-2)k_2^2r+k_2^2r^2h'+(2-n)k_2^2r-2k_2^2r^2h'=\frac{\lambda_{F}m}{4(n-1)}+\frac{\tilde{\lambda}}{2},
\end{equation*}
which implies,
\begin{equation*}
h(r)=\left(\frac{\lambda_{F}m}{4(n-1)k_2^2}+\frac{\tilde{\lambda}}{2k_2^2}\right)r^{-1}+k_1,
\end{equation*}
thus we have,
\begin{equation*}
h(r)=\frac{\lambda_{F}}{2k_2^2}r^{-1}+k_1.
\end{equation*}

Reciprocally, the function $\psi(r)$ and $h(r)$ given in \eqref{23} satisfy \eqref{18}, \eqref{19} and
\eqref{20}. Since
 
\begin{equation*}
h'=-\left(\frac{\lambda_{F}m}{4(n-1)k_2^2}+\frac{\tilde{\lambda}}{2k_2^2}\right)r^{-2},\hspace{0.5cm}h''=2\left(\frac{\lambda_{F}m}{4(n-1)k_2^2}+\frac{\tilde{\lambda}}{2k_2^2}\right)r^{-3}.
\end{equation*}

In this case 
\begin{equation*}
(n-2)\psi''+\psi h''+2\psi'h'= 2\left(\frac{\lambda_{F}m}{4(n-1)k_2^2}+\frac{\tilde{\lambda}}{2k_2^2}\right)rr^{-3} +2 \left(-\left(\frac{\lambda_{F}m}{4(n-1)k_2^2}+\frac{\tilde{\lambda}}{2k_2^2}\right)r^{-2}\right),
\end{equation*}
equivalently,
\begin{equation*}
(n-2)\psi''+\psi h''+2\psi'h'=2\left(\frac{\lambda_{F}m}{4(n-1)k_2^2}+\frac{\tilde{\lambda}}{2k_2^2}\right)r^{-2}-2\left(\frac{\lambda_{F}m}{4(n-1)k_2^2}+\frac{\tilde{\lambda}}{2k_2^2}\right)r^{-2}.
\end{equation*}
Therefore $(n-2)\psi''+\psi h''+2\psi'h'=0$.

Clearly, in this case the second and third equation hold.

If $s=\frac{1}{2}$, we have by \eqref{19} and \eqref{20} that  $\frac{\lambda_{F}}{4(n-1)}\left(m-2n+2\right)+\frac{\tilde{\lambda}}{2}=-\frac{(n-2)k_2^2}{4}$ and $\frac{\lambda_{F}m}{4(n-1)}+\frac{\tilde{\lambda}}{2}=\frac{(n-2)k_2^2}{4}$.
 Consequently we obtain 	$\tilde{\lambda}=-\frac{(m-n+1)(n-2)}{2(n-1)}k_2^2$ and 
 $\lambda_{F}=(n-2)k_2^2$.
  Substituting $\psi(r)=k_2r^{\frac{1}{2}}$ into \eqref{18}, we obtain $h(r)=\frac{(n-2)}{8}\left(\ln r\right)^{2}+c\ln r + c_1$, where $c, c_1\in \mathbb{R}$.

This concludes the proof of Theorem \ref{theorem 4}.
\end{proof}

As a direct consequence of  Theorem \ref{theorem 4} we get,
in the Riemannian case,  the following result of rigidity.

\begin{corollary}\label{coro1}
	Let $\left(\left(M ^{n}, g^{*}\right)\times F^m\right)$, $n\geq3$ and $m\geq 2$ be a complete gradient Schouten soliton, where $(M^n,g^{*})$ is conformal to an Euclidean metric and rotationally symmetric where $g^{*}_{ij}=\delta_{ij}/\psi^{2}$ and $F^m$ is a Einstein Riemannian manifold. Then $\left(M^{n},g^{*}\right)$ is isometric to $\mathbb{S}^{n-1}\times \mathbb{R}$, $F^m$ is a compact manifold with ricci curvature equal to $\lambda_{F}=(n-2)$ and 	$\tilde{\lambda}=-\frac{(m-n+1)(n-2)}{2(n-1)}$,\\
 
\end{corollary}
\begin{proof}
	
Except for homothety we can consider $ k_2 = 1 $ in Theorem \ref{theorem 4}.	It has been proven in \cite{IL} that $(R^n,g^{*})$,
	 $g^{*}_{ij}=\delta_{ij}/\psi^{2}$,  with $\psi(r) = r^s$, is complete only for $ s= \frac{1}{2}$. Moreover,  $\left(\mathbb{R}^{n}\setminus\{0\},g^{*}=\frac{1}{r}g_{0}\right)$ is isometric to $\mathbb{R}\times \mathbb{S}^{n-1}$. Since $\lambda_{F}=(n-2)>0$, it follows from the Bonnet--Myers Theorem that $F$ is compact.
	
\end{proof}

Note that depending on the choice of $m$ and $n$, $\tilde{\lambda}$ can be negative, zero or positive. In this case, it is possible to build examples of complete gradient  schouten soliton, either expanding, steady or  shrinking. We provide three examples  of gradient Schouten solitons.
\begin{example}
	The product manifold 
	 $\left(\mathbb{S}^{2}\times\mathbb{R}\right)\times\left(\mathbb{S}^{2}\times\mathbb{S}^{2}\right)$, with the usual product metric, is a  complete expanding gradient Schouten soliton with $\tilde{\lambda}=-\frac{1}{2}$.
	 In fact, considering $m = 4$ and $n = 3$, and since $S^2 \times S^2$ is an  Eisntein manifold with com $\lambda_F$= 1, it meets the conditions of Theorem \ref{theorem 4}. 
\end{example}

\begin{example}
The product manifold 
$\left(\mathbb{S}^{n-1}\times\mathbb{R}\right)\times\mathbb{S}^{n-1}$, with the usual product metric, is a  complete steady gradient Schouten soliton with $\tilde{\lambda}=0$. In fact, considering $m =n-1 $ and since the unitary sphere $(n-1)$ - dimensional $S^{n-1}$ is an Eisntein manifold with  $\lambda_F= n-2$, it meets the conditions of Theorem \ref{theorem 4}.
\end{example}

\begin{example}
	The product manifold
	$\left(\mathbb{S}^{3}\times\mathbb{R}\right)\times\mathbb{S_R}^{2}$, with the usual product metric, is a  complete shrinking gradient Schouten soliton with $\tilde{\lambda}= \frac{1}{3}$. In fact, considering $m =2 $ and $n=4$,  and since the sphere $S_R^{2}$, with $R= \frac{\sqrt{2}}{2}$, is a  Eisntein manifold with  $\lambda_F= 2$, it meets the conditions of Theorem \ref{theorem 4}.	
\end{example}

\begin{example}
	As a consequence of Corollaly \ref{coro1}  we can build examples of complete gradient Schouten  solitons that are not locally flat. To do this, just consider $((\mathbb{S}^{n-1}\times \mathbb{R}) \times F^m)$ whre $F^m$  is a  non- trivial compact Einstein manifold with ricci curvature equal to $\lambda_{F}=(n-2)$.

\end{example}

\end{document}